\documentclass[12pt]{article}
 \usepackage[margin=1in]{geometry} 
\usepackage{amsmath,amsthm,amssymb,amsfonts}
\usepackage{graphicx}
\usepackage{indentfirst}
\usepackage{mathrsfs}
\usepackage{cleveref}
\usepackage{dsfont} 

\usepackage{hyperref}
\usepackage[affil-it]{authblk}
\usepackage{csquotes}

\newtheorem{thm}{Theorem}[section]
\newtheorem{lem}[thm]{Lemma}
\newtheorem{lemma}[thm]{Lemma}
\newtheorem{prop}[thm]{Proposition}
\newtheorem{cor}[thm]{Corollary}
\newtheorem{corollary}[thm]{Corollary}

\newcommand{\N}{\mathbb{N}}
\newcommand{\Z}{\mathbb{Z}}

\renewcommand{\P}{\mathbb P}

\newcommand{\C}{\mathbb C}

\newcommand{\EX}{\mathbb E}

\newcommand{\D}{\mathbb D}
\newcommand{\Q}{\mathbb Q}

\newcommand{\<}{\leq}
\renewcommand{\>}{\geq}

\newcommand{\2}{\alpha}
\newcommand{\3}{\beta}
\renewcommand{\r}{\gamma}
\newcommand{\e}{\varepsilon}
\newcommand{\8}{\infty}
\newcommand{\0}{\theta}
\renewcommand{\O}{\varnothing}
\newcommand{\6}{\partial}
\renewcommand{\-}{\setminus}

\newcommand{\w}{\omega}
\renewcommand{\d}{\delta}

\renewcommand{\k}{\kappa}
\renewcommand{\t}{\tau}
\renewcommand{\a}{\sigma}

\newcommand{\inn}{\subseteq}

\renewcommand{\~}{\tilde}

\DeclareMathOperator*{\esup}{esup}

\newcommand\blfootnote[1]{%
  \begingroup
  \renewcommand\thefootnote{}\footnote{#1}%
  \addtocounter{footnote}{-1}%
  \endgroup
}

\begin{document}

\title{Liouville heat kernel upper bounds at large distances}
\author{Yang Yu
	\thanks{University of Washington, Seattle, WA, USA; \texttt{yuy10@uw.edu}}}
\date{}

\maketitle

\begin{abstract}
We show that the Liouville heat kernel decays fast at large distances. In particular, the Liouville semigroup $T_t$ is $C_0$-Feller, where $C_0$ is the space of real-valued continuous functions on $\C$ vanishing at infinity. This is a problem mentioned in the paper \cite{andres2016continuity}.
\end{abstract}

\begin{NoHyper}
\end{NoHyper}

\section{Introduction}

Liouville quantum gravity (LQG) was introduced by Polyakov in a seminal paper \cite{polyakov1989quantum} and can be considered as the canonical 2-dimensional random Riemannian manifold. The Riemannian volume form can be  formally written in the form
$$e^{\r X(z)}dz$$
where $X$ is a massive Gaussian free field (GFF) on $\C$; $\r\in(0, 2)$ is a parameter; and $dz$ is the Lebesgue measure on $\C$. 

Of course the above form is not rigorous as the GFF is not a random function (but a distribution in the sense of Schwartz). Nonetheless, one can make sense of the volume form by the theory of Gaussian multiplicative chaos \cite{kahane1985chaos} or some other regularization procedure \cite{duplantier2011liouville}. The rigorous construction of the random volume form is then referred as the Liouville measure $M_\r$.

The Liouville Brownian motion (LBM) is the canonical diffusion process for Liouville quantum gravity, which is constructed in \cite{garban2016liouville, Berestycki2013DiffusionIP} as a time-changed Brownian motion of 2-dimension according to the Liouville measure (independent of the Brownian motion). More precisely, for the Liouville measure $M_\r$ one can construct the associated positive continuous additive functional (PCAF) $F$ of a Brownian motion $W$ which can be formally written as
$$F_t=\int_0^t e^{\r X(W_s)}\,ds.$$
Then the LBM $\{Y_t\}_{t\>0}$ as a stochastic process is defined by 
$Y_t:=W_{F^{-1}_t}$, where $F^{-1}$ is the inverse of $F$ (the inverse exists). For the rigorous discussion of LBM one can refer to \cite{garban2016liouville, Berestycki2013DiffusionIP,fukushima2010dirichlet} or see \Cref{setup} of this paper.

The heat kernel of Liouville Brownian motion (LHK) is constructed in \cite{rhodes2014}. Further properties of LHK are studied in \cite{andres2016continuity,maillard2016liouville,ding2019heat}. However none of them indicates large distance behavior of LHK.

In \cite{garban2016liouville} it is shown that the semigroup $T_t$ of Liouville Brownian motion is weak Feller, meaning that the semigroup operator $T_t$ maps bounded continuous functions to bounded continuous functions. In \cite{rhodes2014} they show $T_t$ is strong Feller, meaning that $T_t$ maps bounded Borel measurable functions to continuous functions. But it is not clear whether it is $C_0$-Feller. That is, we don't know $T_t(C_0)\inn C_0$, where $C_0$ is the space of continuous functions vanishing at infinity (it is also mentioned in \cite[Remark 2.3]{andres2016continuity}). This is one of the motivations for this paper.

In this paper we show that the LHK decays fast at large distances (\Cref{mainthm}), which immediately implies $C_0$-Feller property. We also attach in Appendix a simple proof of Feller property without using estimates of LHK.

\section{Background and preliminaries}\label{setup}

\subsection{The massive Gaussian free field and the Liouville measure}

Given a real number $m>0$, the whole-plane massive Gaussian free field (MGFF) (see \cite{sheffield2007gaussian} for more information about Gaussian free field) $X$ is a centered Gaussian random distribution (in the sense of Schwartz) with covariance function given by the Green function of the operator $m^2-\Delta$, that is,

$$\EX[X(x)X(y)]=G_{m}(x, y)=\int_{0}^{\infty} e^{-\left(m^{2} / 2\right) u-|x-y|^{2} /(2 u)} \frac{d u}{2 u} \quad  \text{ for all } x,y\in \C.$$
Note that $G_m(x,y)$ can be written as $$G_{m}(x, y)=\int_{1}^{+\infty} \frac{k_{m}(u(x-y))}{u} d u$$
where $k_{m}(z)=\frac{1}{2} \int_{0}^{\infty} e^{-\frac{m^{2}}{2 s}|z|^{2}-\frac{s}{2}} d s$ is a continuous covariance kernel
(see \cite{allez2013lognormal} for details about this expression). 
This expression helps us to decompose $X$ into a sum of good Gaussian fields.

We then introduce the $n$-regularized field $X_n$. For this purpose, let $\{c_n\}_{n\in\N}$ be a strictly increasing sequence of real numbers starting from $c_0=1$ and satisfying $\lim_{n\to\8}c_n=\8$. Let $(\eta_n)_{n\geq1}$ be a family of independent continuous Gaussian fields on $\C$ with covariance 
$$\EX[\eta_n(x)\eta_n(y)]=\int_{c_{n-1}}^{c_n} \frac{k_{m}(u(x-y))}{u} d u\quad  \text{ for all } x,y\in \C.$$
Note that for each $n$ we can choose $\eta_n$ to be continuous in space by applying Kolmogorov continuity theorem (\cite[Theorem 2.23]{kallenberg2006foundations}).
Define $X_n:=\sum_{k=1}^n \eta_k$, and the associated random Radon measure $M_n=M_{\r,n}$ on $\C$ by 
$$M_{n, \gamma}(d z)=\exp \left(\gamma X_{n}(z)-\frac{\gamma^{2}}{2} \EX\left[X_{n}(z)^{2}\right]\right) d z, \quad \gamma \in[0,\8)$$
where $dz$ is the Lebesgue measure on $\C$. By Kahane's theory of multiplicative chaos \cite{kahane1985chaos} almost surely $M_n$ converges vaguely toward a limit Radon measure $M$, which is called the Liouville measure. The law of the limit does not depend on the choice of ${c_n}$ and the limit measure is nontrivial if and only if $\r\in[0,2)$. 

Recall (see \cite{kahane1985chaos,rhodes2014GMC}) that the Liouville measure has an important property that for any bounded Borel set $A$ and $p\in (-\8,4/\r^2)$ we have $\EX [M(A)^{p}]<\8$ and that 
$$\sup _{r\in(0,1]} r^{-\xi_{M}(p)} \mathbb{E}\left[M(r A)^{p}\right] \leq C_{p}$$
for some constant $C_p$ only depending on $p$, $\text{diam}A(:=\sup_{x,y\in A}|x-y|)$, $\r$ and $m$, where $\xi_M(q)=-\frac{\r^2}{2}q^2+(2+\frac{\r^2}{2})q$ is the power law spectrum of $M$ (see \cite{allez2013lognormal}).

\subsection {Liouville Brownian motion}
The Liouville Brownian motion is constructed in \cite{garban2016liouville, Berestycki2013DiffusionIP} as the canonical diffusion process under the geometry induced by the measure $M$. More precisely, Let $(\Omega, \mathcal A, \P)$ be the probability space that $(\eta_n)_{n\geq1}$ live on. Let $\Omega':=C([0,\8),\C)$ and $W=(W_t)_{t\geq0}$ be the coordinate procress on $\Omega'$. Set $\mathcal F=\a(W_s,s<\8)$ and $\mathcal F_t=\a(W_s,s\<t)$. Let $\{P_x\}_{x\in\C}$ be the family of probability measures on $(\Omega',\mathcal F)$ such that $W$ under $P_x$ is a Brownian motion on $\C$ starting from $x\in\C$.

For each $n\in\N$ define $F^n(t):\Omega\times\Omega'\to[0,\8)$ to be 
$$F^n(t):=\int_0^t \exp \left(\gamma X_{n}(W_s)-\frac{\gamma^{2}}{2} \EX\left[X_{n}(W_s)^{2}\right]\right) ds, \quad t\geq0.$$
Note that $F^n(t)$ is the positive continuous additive functional (\cite{chen2012symmetric}, \cite{fukushima2010dirichlet}) of $W$ with Revuz measure $M_{\r,n}$.
In \cite[Theorem 2.7]{garban2016liouville} (see also \cite[Theorem 1.2]{Berestycki2013DiffusionIP}) they show that $\P$-a.s. there exists a unique positive continuous additive functional (PCAF) $F=(F(t))_{t\geq 0}$ of $W$ such that the Revuz measure of $F$ is $M$ and
$$\lim_{n\to\8} P_x[\sup_{t\leq T}|F^n(t)-F(t)|>\e]=0\quad \text{ for all }\e>0,T>0, x\in \C. $$
And then the Liouville Brownian motion is defined to be 
$$Y_t=W_{\bar F(t)}$$
where $\bar F(t)=F^{-1}(t)=\inf\{s\geq 0:F(s)>t\}$. Note that it is proved in \cite{garban2016liouville} (see also \cite[Theorem 1.2]{Berestycki2013DiffusionIP}) that $\P$-a.s. for any $x\in\C$, $P_x$-a.s., $F$ is continuous, strictly increasing and diverging to $\8$.

\subsection{Notation}
Throughout this paper, we will fix $\r\in(0,2)$. Define two constants in terms of $\r$ which we will frequently use: $\2_1=\frac{1}{2}(2+\r)^2$, $\2_2=\frac{1}{2}(2-\r)^2$. Let $X$ be a massive GFF on $\C$ and $M=M_\r$ be the Liouville measure constructed from $X$. We write $$\~\xi(q)=-\xi_M(-q)=(2+\frac{\r^2}{2})q+\frac{\r^2}{2}q^2$$
for $q>0$. Let $\{Y_{t}\}_{t\>0}$ be a LBM and $p_t(x,y)$ be its heat kernel w.r.t. the Liouville measure $M$. 

We denote $a\vee b=\max\{a,b\}$ and $a\wedge b=\min\{a,b\}$. Let $B_{x,r}=\{z\in\C:|z-x|\leq r\}$, in particular we write $B_R=\{z\in\C:|z|\leq R\}$. Let $\t_{x,r}=\inf\{t\>0:Y_t\notin B_{x,r}\}$ be the first exit time of LBM of the ball $B_{x,r}$.

The symbols $c,C$ stand for positive constants whose value may change from line to line, but they won't depend on any parameters in this article. By adding subscripts $X,\r,\2,...$ to the symbols $c,C$ we indicate their dependence on those subscripts, while some other symbols $\bar C_R,\hat C_R, C_*,...$ are exclusively used in some propositions or theorems.

We use $\lesssim$ to indicate the inequality holds up to an absolute constant $C>0$, i.e. $x\lesssim y$ if and only if $x\<Cy$ for some $C>0$. By adding subscripts $X,\r,\2,...$ to the symbols $\lesssim$ we indicate dependence of the constant on those subscripts. 

We use $P_x,E_x$ to take the probability (expectation) w.r.t. the Brownian motion starting at $x\in\C$, and use $\P, \EX$ to take the probabitlity (expectation) w.r.t. the massive GFF.

\section{The estimates}

In this section we will establish some estimates of the Liouville heat kernel (LHK).

\subsection{Liouville measure at large distances}

We first do some preparation for estimates of LHK. The following lemma will be used in \Cref{LMbound}, \Cref{liouvillemeasure}, and \Cref{exitmoment}.

\begin{lemma} \label{moment}
	Let $\{Z_R\}_{R\>1}$ be a family of nonnegative random variables that are almost surely  nondecreasing in $R$ such that 
	$$\EX [Z_R^p]\<CR^m\quad \text{ for all } R\>1.$$
	for some positive constants $p,m,C>0$. Then for any $\0>m/p$, almost surely there is a random constant $C_\0>0$ such that $Z_R\<C_\0 R^{\0}$ for all $R\>1$.
\end{lemma}

\begin{proof}
	Let $R_n=2^{n}$ and for any $\0>0$ define $A_n=\{Z_{R_n}\<R_n^{\0}\}$. Then 
	$$\P[A_n^{c}]\leq R_n^{-\0 p}\EX Z_{R_n}^p\<C R_n^{m-\0 p}\quad \text{ for all } n\>0.$$
	If $\0>m/p$ then by Borel-Cantelli's lemma $\P[A_n^c \text{ i.o.}]=0$ and thus almost surely there is a random constant $C_\0>0$ such that $Z_{R_n}\<C_\0 R_n^\0$ for any $n\geq0$. By monotonicity we have 
	$$Z_R\leq Z_{R_{n+1}}\leq C_\0 R_{n+1}^\0= C_\02^\0 R_n^\0\leq C_\02^\0 R^\0$$
	provided $R_n\leq R\leq R_{n+1}$ for $n\>0$. Reassigning $C_\02^\0$ as $C_\0$ finishes the proof. 
\end{proof}

The following proposition gives the Liouville volume growth rate of Euclidian balls which will be used in \Cref{DUE}.

\begin{prop}\label{LMbound}
	For any $\e>0$, $\P$-a.s. the Liouville measure satisfies
	$$M(B_R)\lesssim_{X,\r,\e}R^{2+\e}\qquad \text{ for all } R\>1.$$
\end{prop}
\begin{proof}

Notice that for any $n$-regularized Liouville measure $M_n$ and bounded Borel set $A$ we have 
$$\EX [M_{n}(A)]=\int_A \EX\exp \left(\gamma X_{n}(z)-\frac{\gamma^{2}}{2} \EX\left[X_{n}(z)^{2}\right]\right) d z=\int_A 1 \,dz .$$
Hence letting $n\to\8$, by vague convergence (in fact we have $M_n(A)\to M(A)$) and Fatou's Lemma we have
  $$\EX [M(B_R)]\<\EX [\varliminf_{n\to\8} M_n(B_R)]\< \varliminf_{n\to\8}\EX [M_n(B_R)]=\pi R^2.$$ Then apply \Cref{moment} to get the bound.
\end{proof}

Next we give the growth rate of the coefficients of H\"older continuity of the Liouville measure, which will be also used in \Cref{DUE}.

\begin{prop}\label{liouvillemeasure}
	For any $\r\in(0,2)$ and $\2\in(0,\2_2)$, set $m_0(\r,\2)=\frac { \gamma ^ { 2 } } { 2 } + \frac { 4\2 \gamma ^ { 2 } } { ( \2_1 - \alpha ) ( \2_2 - \alpha ) }$. Then there exists a random constant $\bar C_R$ depending on $X,\r,\2,R$ such that $\P$-a.s. 
	$$\sup_{|x|\leq R}M(B_{x,r})\leq \bar C_Rr^\2\quad  \text{ for all } r\in(0,1]$$
	and for any $\e>0$ and any $R\geq1$ we have
		$$\bar C_R\lesssim_{X,\r,\2,\e} R^{m_0+\e}.$$ 
\end{prop}

\begin{proof}
We prove it in a similar manner as in \cite[Theorem 2.2]{garban2016liouville}, but give the coefficient estimates depending on $R$. The main idea is to improve Borel-Cantelli's lemma and use stationarity of the Liouville measure.

	For $n\in\N$ we partition $[-8,8]^2$ into $2^{2n}$ dyadic squares $\{I_n^j:j=1,2,...,2^{2n}\}$ of equal size. Fix $\2>0$, let $A_n$ be the event that $M(I_n^j)\leq2^{-\2n}$ for all $1\leq j\leq 2^{2n}$. Then for $p\in(0,4/\r^2)$ we have using the stationarity of GFF and the power law of the Liouville measure 
\begin{align*} \mathbb{P}[A_n^c] & \leq 2^{p\alpha n} \mathbb{E}\left[\sum_{1 \leq j \leq 2^{2 n}} M\left(I_{n}^{j}\right)^{p}\right] \\ & \leq  C_p2^{-nK(\2,p)}
\end{align*}
	where $K(\2,p):=\xi_M(p)-\2p-2$. Set $E_n=\cap_{k=n}^\8A_k$ and $\~ E_0=E_0$, $\~E_n=E_n\-E_{n-1}$ for $n\in\N^*$, then $\P[\~E_n]\leq\P[A_{n-1}^c]$ for $n\in\N^*$,  and $\~E_n$ are disjoint and $\P[\cup_{n=0}^\8\~E_n]=\P[\cup_{n=0}^\8E_n]=1-\P[A^c_n \text{ i.o.}]=1$ by Borel-Cantelli's lemma. 
	
	Define
	$$\bar C_0:=\begin{cases}
		4 					& \text{on $\~E_0$}\\
		4\vee\sup_{|x|\leq1,r\in(2^{-n},2)}\frac{M(B_{x,r})}{r^\2} & \text{on $\~E_n$ for $n\in\N^*$.}\\
	\end{cases}$$
Note that $\bar C_0$ is almost surely well-defined because $\~E_n$ are disjoint and $\P[\cup_{n=0}^\8\~E_n]=1$. Also $\bar C_0$ is $\mathcal A$-measurable as $\sup_{|x|\leq1,r\in(2^{-n},2)}\frac{M(B_{x,r})}{r^\2}=\sup_{x\in\Q^2, |x|\leq1,r\in(2^{-n},2)}\frac{M(B_{x,r})}{r^\2}$. This is because for $x\notin \Q^2$ with $|x|\<1$ and $r\in(2^{-n},2)$ we can find $x_i\in \Q^2$ with $|x_i|\<1$ and $r_i \in (2^{-n},2)$ such that $x_i\to x$, $r_i\downarrow r$, $B_{x,r}\inn B_{x_i,r_i}$ and hence $\frac{M(B_{x,r})}{r^\2}\<\limsup_{i\to\8}\frac{M(B_{x_i,r_i})}{r_i^\2}$.

We claim $\P$-a.s. $M(B_{x,r})\leq \bar C_0r^\2$ for any $|x|\leq1$ and $r\in(0,1] $. Indeed, on $\~E_n$, when $r\in(2^{-n},1]$ by the definition of $\bar C_0$ we have $M(B_{x,r})\leq\bar C_0r^\2$; when $r\in(2^{-k-1},2^{-k}]$ for $k\geq n$,
	any ball $B_{x,r}$ is contained in at most 4 dyadic squares $I_{k+1}^j$ and each square $I_{k+1}^j$ (of size $2^{3-k}$) has Liouville measure no greater than $2^{-(k+1)\2}$, hence $M(B_{x,r})\leq 4\cdot2^{-\2(k+1)}\leq4r^\2$.
	
	Moreover, for $\0>0$ by H\"older inequality for $q^{-1}+q'^{-1}=1$
\begin{align*}
\EX \bar C_0^\0&\leq 4^\0+\sum_{n=1}^\8 2^{n\2\0}\EX[M(B_3)^\0;\~E_n]\\
&\leq 4^\0+\sum_{n=1}^\8  2^{n\2\0}\EX[M(B_3)^{\0q'}]^{1/q'}\P[\~E_n]^{1/q} \\
&\leq 4^\0+\sum_{n=1}^\8  2^{n\2\0}\EX[M(B_3)^{\0q'}]^{1/q'}(C_p2^{-(n-1)K(\2,p)})^{1/q}\\
&= 4^\0+C_p^{1/q}\EX[M(B_3)^{\0q'}]^{1/q'}\sum_{n=1}^\8 2^{-n(K(\2,p)/q-\2\0)+K(\2,p)/q}.
\end{align*}
The above is finite if $\0q'<4/\r^2$ and $K(\2,p)/q-\2\0>0$. So
$\0<\frac{K(\2,p)}{q\2}\wedge\frac{4}{\r^2q'}$. Take $p=\frac{2+\r^2/2-\2}{\r^2}$ $(<\frac{4}{\r^2})$ (whence $K(\2,p)=\frac{({\2_1}-\2)({\2_2}-\2)}{2\r^2}$) and $q=\frac{({\2_1}-\2)({\2_2}-\2)}{8\2}+1$ to maximize the right hand side to get $\EX \bar C_0^\0<\8$ whenever $\0<\frac{({\2_1}-\2)({\2_2}-\2)}{2\r^2(({\2_1}-\2)({\2_2}-\2)/8+\2)}$. 

Now do the same partition and reasoning for each region $z_k+[-8,8]^2$ where $z_k\in\Z^2$ and we get a sequence of $\bar C_{z_k}$ (defined similar to $\bar C_0$) with the same distribution as $\bar C_0$. Set $\bar C_R=\max_{z_k\in\Z^2\cap B_{R+1}}\bar C_{z_k}$. Since any ball $B_{x,r}$ with $|x|\leq R$ and $r\in(0,1]$ is contained in one of the regions $\{z_k+[-8,8]^2\}_{z_k\in\Z^2\cap B_{R+1}}$ (one can find $z_k\in \Z^2$ with $|x-z_k|\<1$ for each $x\in \C$ with $|x|\<R$), thus $\sup_{|x|\leq R}M(B_{x,r})\leq \bar C_Rr^\2$ for any $r\in(0,1]$ . Moreover, using union bound and the stationarity of GFF, we have for some absolute constant $C>0$ that
$$ \EX\bar C_R^\0\< \sum_{z_k\in \Z^2 \cap B_{R+1}} \EX \bar C_{z_k}^\0\<CR^2\EX \bar C_0^\0.$$
By \Cref{moment}, we can show that $\bar C_R\lesssim_{X,\r,\2,m_1} R^{m_1}$ for $R\geq1$ when $m_1>2/\0$. Combining with the bound for $\0$, we get $m_1>m_0(\r,\2):=\frac { \gamma ^ { 2 } } { 2 } + \frac { 4\2 \gamma ^ { 2 } } { ( \2_1- \alpha ) ( \2_2- \alpha ) }$. 
\end{proof}

~\\

It is natural to ask whether we can get similar estimates for the lower bound coefficients. Here we give the estimates but with some cost on the range of lower H\"older exponent $\2$. We won't use the following proposition in the rest of this paper.

\begin{prop}\label{LMLB}
	For any $\r\in(0,2)$ and $\2>\r^2/2+2\sqrt 2\r+2~(>\2_1)$, there is $m_{00}(\r,\2)>0$ such that, there exists a random constant $\bar{c}_R$ depending on $X,\r,\2,R$ such that $\P$-a.s. 
		$$\inf_{|x|\leq R}M(B_{x,r})\geq \bar{c}_Rr^\2\quad  \text{ for all } r\in(0,1]$$
	and for any $\e>0$ and any $R\geq1$ we have
		$$\bar{c}_R\gtrsim_{X,\r,\2,\e} R^{-m_{00}-\e}.$$

\end{prop}

\begin{proof}
	
	For each $n\in \N$ we partition $[-1,1]^2$ into $2^{2n}$ dyadic squares $\{I_n^j:j=1,2,...,2^{2n}\}$ of equal size and define good events 
	$$A_n=\{\inf_{1\<j\<2^{2n}} M(I_n^j)\>2^{-\2n}\}$$
	and set $E_n=\cap_{k=n}^\8A_k$ and $\~E_0=E_0$, $\~E_n=E_n\-E_{n-1}$ for $n\in\N^*$. Note that $\~E_n\inn A_{n-1}^c$.
	 Then for $p<0$ by using Markov inequality, the stationarity of GFF and the power law of the Liouville measure we have
\begin{align*} \P[\~E_{n+1}]\<\mathbb{P}[A_n^c] & \leq 2^{p\alpha n} \mathbb{E}\left[\sum_{1 \leq j \leq 2^{2 n}} M\left(I_{n}^{j}\right)^{p}\right] \\ & \leq  C_p2^{-nK(\2,p)}
\end{align*}
	where $K(\2,p):=\xi_M(p)-\2p-2$. When $K(\2,p)>0$ by Borel-Cantelli's lemma we have  
 $$\P[\cup_{n=0}^\8 \~E_n]=\P[\cup_{n=0}^\8E_n]=1-\P[A_n^c \text{ i.o.}]=1.$$
	Define
	$$\bar{c}:=\begin{cases}
		8^{-\2} 					& \text{on $\~E_0$}\\
		8^{-\2}\wedge\inf_{|x|\leq1,r\in(2^{-n},1]}\frac{M(B_{x,r})}{r^\2} & \text{on $\~E_n$ for $n\in\N^*$.}\\
	\end{cases}$$
Note that $\bar{c}$ is almost surely well-defined because $\~E_n$ are disjoint and $\P[\cup_{n=0}^\8 \~E_n]=1$. The $\mathcal A$-measurability of $\bar c$ can be shown in a similar way to the proof of the $\mathcal A$-measurability of $\bar C_0$ in \Cref{liouvillemeasure}.
	We claim $\P$-a.s. $M(B_{x,r})\geq \bar{c}r^\2$ for any $|x|\leq1$ and $r\in(0,1] $. Indeed, on $\~E_n$, when $r\in(2^{-n},1]$ by the definition of $\bar{c}$ we have $M(B_{x,r})\geq\bar{c}r^\2$; when $r\in(2^{-k-1},2^{-k}]$ for $k\geq n$,
	any ball $B_{x,r}$ contains at least 1 dyadic square $I_{k+3}^j$ and each square $I_{k+3}^j$ (of size $2^{-2-k}$) has Liouville measure no less than $2^{-(k+3)\2}$, hence $M(B_{x,r})\> 2^{-\2(k+3)}\>8^{-\2} r^\2$.
	
	Moreover, for $\0>0$ by H\"older inequality for $q^{-1}+q'^{-1}=1$
\begin{align*}
\EX (\bar{c})^{-\0} 
&\leq 8^{\2\0}+\sum_{n=1}^\8 \EX[\sup_{|x|\leq1,r\in(2^{-n},1]}r^{\2\0}M(B_{x,r})^{-\0};\~E_n]\\
&\leq 8^{\2\0}+\sum_{n=1}^\8 \EX[\sup_{1\<j\<2^{2n+4}}M(I_{n+2}^{j})^{-\0};\~E_n]\\
&\leq 8^{\2\0}+\sum_{n=1}^\8  2^{2n+4}\EX[M(I_{n+2}^{1})^{-\0q'}]^{1/q'}\P[\~E_n]^{1/q} \\
&\leq 8^{\2\0}+C_{\r,p,q,\0}\sum_{n=1}^\82^{2n}\cdot2^{-n\xi(-\0q')/q'}\cdot2^{-(n-1)K(\2,p)/q}
\end{align*}
The above is finite if $$2-K(\2,p)/q-\xi(-\0 q')/q'<0.$$
Solving the above inequality, we have 
$$\0<-\frac{\xi^{-1}(q'(2-K(\2,p)/q))}{q'}$$
where $\xi^{-1}(x)=\frac{1}{2}+\frac{2}{\r^2}-\frac{1}{\r^2}\sqrt{(2+\r^2/2)^2-2\r^2x}$ by the quadratic formula.
Since $\0>0$ we need $2-K(\2,p)/q<0$. Set $p=p(\2)=\frac{2+\r^2/2-\2}{\r^2}(<0)$, and noting that $q$ can be chosen arbitrarily close to 1, we have $\2>2+\frac{\r^2}{2}+2\sqrt{2}\r$.

Now do the same partition and reasoning for each region $z_k+[-1,1]^2$ where $z_k\in\Z^2$ and we get a sequence of $\bar{c}_{z_k}$ (defined similar to $\bar{c}$) with the same distribution as $\bar{c}$. Set $\bar{c}_R=\min_{z_k\in\Z^2\cap B_{R+1}}\bar c_{z_k}$. Since for any ball $B_{x,r}$ with $|x|\leq R$ and $r\in(0,1]$, one can find $z_k\in \Z^2\cap B_{R+1}$ with $|x-z_k|\<1$, we have $\inf_{|x|\leq R}M(B_{x,r})\geq \bar{c}_R r^\2$ for any $r\in(0,1]$. Moreover, using union bound and the stationarity of GFF, we have for some absolute constant $C>0$ that
\begin{align*}
 \EX\bar{c}_R^{-\0}
 &\<\sum_{z_k\in\Z^2\cap B_{R+1}}\EX\bar{c}_{z_k}^{-\0}\<CR^2\EX \bar{c}^{-\0}.
\end{align*}
By \Cref{moment}, we can show that $\bar{c}_R\gtrsim_{X,\r,\2,m} R^{-m}$ for $R\geq1$ when $m>2/\0$. Combining with the bound for $\0$ we get $$m>2/\0>2q'\r^2/\left(\sqrt{(2+\r^2/2)^2-4q'\r^2+\frac{q'}{q}(\2_2-\2)(\2_1-\2)}-\frac{\r^2}{2}-2\right)=:m_{00}(\r,\2)$$
where we have chosen some $q=q(\2)$ such that $2-K(\2,p(\2))/q<0$.
\end{proof}

\subsection{Exit time estimates}

\begin{lem} \label{exitmoment}
	For any $\r\in(0,2),~q>0,~p>1,p':=p/(p-1),~\kappa>p(2+\~\xi(q))$, and any $\e>0$, there exists a random constant $\hat C_R$ depending on $X,\r,q,\k,R$ such that $\P$-a.s. 
	$$\sup_{|x|\leq R}E_x[\tau_{x,r}^{-q}]\leq \hat C_R r^{-\kappa}\quad  \text{ for all } r\in(0,1],$$
	and for any $R\geq1$,
	$$\hat C_R\lesssim_{X,\r,q,p,\kappa,\e}R^{2p'+\e}.$$ 
\end{lem}

\begin{proof}
We follow the proof in \cite[Proposition 3.2]{andres2016continuity}, but give the coefficient estimates depending on $R$. The main idea is the same as \Cref{liouvillemeasure}, i.e., to improve Borel-Cantelli's lemma and use the stationarity of the Liouville measure.

Let $\mu_{y,r}^z$ be the harmonic measure of the circle $\6 B_{y,r}$ viewed at $z\in\C$. In particular when $z=y$, $\mu_{y,r}^z$ is the uniform distribution on $\6 B_{y,r}$ and we set $\mu_{y,r}=\mu_{y,r}^y$. When $|z-y|\leq r/2$ we have $\mu_{y,r}^z\leq C\mu_{y,r}$ for some absolute constant $C>0$.
For $n\in\N$, set $r_n:=2^{-n}$ and $\Xi_n:=\{(i2^{-n},j2^{-n}):i,j\in [-2^{n},2^{n}]\cap\Z\}$. 
In the proof of \cite[Proposition 3.2]{andres2016continuity}, they obtained that
	$$\mathbb{E} E_{\mu_{x,r_n}}\left[\tau_{x,2r_n}^{-q}\right] \leq C_{\r,q} r^{-\tilde{\xi}(q)}.$$ 
	Define the event
	$$A_n:=\{\max_{x\in\Xi_{n+1}} E_{\mu_{x,r_n}}[\tau_{x,2r_n}^{-q}]\leq r_n^{-\kappa}\},\quad E_n:=\cap_{k=n}^\8A_k$$ and $\~E_0:=E_0$, $\~E_n:=E_n\-E_{n-1}$ for $n\in\N^*$. For $n\in\N$ we have using the stationarity of GFF and the power law of the Liouville measure that  
	\begin{align*}
\P[\~E_{n+1}]\leq \P[A_n^c]
&\leq r_n^{\kappa}\sum_{x\in\Xi_{n+1}}\EX E_{\mu_{x,r_n}}[\tau_{x,2r_n}^{-q}]\\
&\leq r_n^{\kappa}\cdot (2^{n+1}+1)^2C_{\r,q}r_n^{-\~\xi(q)} \\
&\<9C_{\r,q}r_n^{\kappa-\~\xi(q)-2}. 
\end{align*}
By Borel-Cantelli's lemma $\P[\cup_{n=0}^\8 \~E_n]=\P[\cup_{n=0}^\8E_n]=1-\P[A_n^c \text{ i.o.}]=1$.

Now define 
\[\hat C_0:=\begin{cases}
			C8^{\kappa} &\text{on $\~E_0$},\\
			C\left( 8^{\kappa}\vee\max_{x\in \Xi_{n+3}}E_{\mu_{x,r_{n+2}}}[\tau_{x,r_{n+1}}^{-q}] \right) &\text{on $\~E_n$ for $n\in\N^*$}.\\
			\end{cases}
\]
Note that $\hat C_0$ is well-defined because $\~E_n$ are disjoint and $\P[\cup_{n=0}^\8 \~E_n]=1$.
We claim 
$$E_x[\tau_{x,r}^{-q}]\leq\hat C_0r^{-\kappa}$$ 
for all $x\in B_1$ and $r\in(0,1]$. Indeed, fix $n_0\in\N$. When $r\in(2^{-n+2}, 1]$ for $n\geq n_0+3$, we have for any $x\in B_1$, there is some $x_i\in\Xi_{n+1}$ such that $|x-x_i|\leq r_{n+1}$. By the strong Markov property
$$
 E_x\tau_{x, r}^{-q}
\leq E_{\mu^x_{x_i, r_n}}[\tau_{x,r}^{-q}]
\leq E_{\mu^x_{x_i, r_n}}[\tau_{x_i,2r_n}^{-q}]
\leq CE_{\mu_{x_i, r_n}}[\tau_{x_i,2r_n}^{-q}],
$$
and this is at most $\hat C_0\< \hat C_0 r^{-\k}$ on $\~E_{n-2}$. Thus the claim holds on $\~E_{n_0}$ for $r\in (2^{-n_0},1]$.
Moreover on $\~E_{n_0}\inn E_n$, if $r\in (2^{-n+2},2^{-n+3}]$ we have 
$$ E_x\tau_{x,r}^{-q}\< CE_{\mu_{x_i, r_n}}[\tau_{x_i,2r_n}^{-q}]\leq Cr_n^{-\kappa}\leq C8^{\kappa}r^{-\kappa}\leq\hat C_0r^{-\kappa}.$$
Hence the claim is true.

Next we examine the moment of $\hat C_0$. Let $p>1$ and $p'=\frac{p}{p-1}$. Then

\begin{align*}
\EX \hat C_0^{1/p'}
&\leq (C8^{\k})^{1/p'}+\sum_{n=1}^\8 \EX\left[\left(C\sup_{x\in \Xi_{n+3}} E_{\mu_{x,r_{n+2}}}[\tau_{x,r_{n+1}}^{-q}]\right)^{1/p'};\~E_n\right]\\
&\leq (C8^{\k})^{1/p'}+\sum_{n=1}^\8\left[C\EX\sup_{x\in \Xi_{n+3}} E_{\mu_{x,r_{n+2}}}[\tau_{x,r_{n+1}}^{-q}]\right]^{1/p'}\P[\~E_n]^{1/p}\\
&\leq (C8^{\k})^{1/p'}+\sum_{n=1}^\8\left[C\EX\sum_{x\in \Xi_{n+3}} E_{\mu_{x,r_{n+2}}}[\tau_{x,r_{n+1}}^{-q}]\right]^{1/p'}\P[\~E_n]^{1/p}.\\
\end{align*}
Using the stationarity of GFF and the bounds for $\EX E_{\mu_{x,r_{n+2}}}[\tau_{x,r_{n+1}}^{-q}]$ and $\P[\~E_n]$, we get 
\begin{align*}
\EX \hat C_0^{1/p'}
&\<  (C8^{\k})^{1/p'}+\sum_{n=1}^\8(C2^{2(n+5)})^{1/p'}\left[\EX E_{\mu_{x,r_{n+2}}}[\tau_{x,r_{n+1}}^{-q}]\right]^{1/p'}\P[\~E_n]^{1/p}\\
&\leq (C8^{\k})^{1/p'}+\sum_{n=1}^\8(C2^{2(n+5)})^{1/p'}(C_{\r,q}r_{n+2}^{-\~\xi(q)})^{1/p'}(9C_{\r,q}r_n^{\kappa-\~\xi(q)-2})^{1/p}\\
&=  (C8^{\k})^{1/p'}+C_{\r,q,p,\k}\sum_{n=1}^\8 r_n^{\frac{1}{p}(\k-2)-\frac{2}{p'}-\~\xi(q)}.
\end{align*}
When $\frac{1}{p}(\k-2)-\frac{2}{p'}-\~\xi(q)>0$, i.e. $\k>p(2+\~\xi(q))$, we have $\EX \hat C_0^{1/p'}
<\8$.

Now do the same partition and reasoning for each region $z_k+[-1,1]^2$ where $z_k\in\Z^2$ and we get a sequence of $\hat C_{z_k}$ (defined similar to $\hat C_0$) with the same distribution as $\hat C_0$. For $R\>1$ set $\hat C_R:=\max_{z_k\in\Z^2\cap B_{R+1}}\hat C_{z_k}$. Then 
$$\sup_{|x|\leq R}E_x[\tau_{x,r}^{-q}]\leq \hat C_R r^{-\kappa}\quad  \text{ for all } r\in(0,1].$$
Moreover
\begin{align*}
\EX[\hat C_R^{1/p'}]
\< \sum_{z_k\in \Z^2\cap B_{R+1}}\EX[\hat C_{z_k}^{1/p'}]
\<CR^2\EX [\hat C_0^{1/p'}].
\end{align*}
By \Cref{moment}, we can show that for any $\e>0$ we have $\P$-a.s. $\hat C_R\lesssim_{X,\r,q,p,\k,\e} R^{2p'+\e}$ for $R\geq1$.
\end{proof}

\begin{cor}
Let $q,p,p',\k$ and $\hat C_R$ be as in \Cref{exitmoment}. For any $\3>\kappa/q$ and $\e\in(0,1)$, with $\d_R:=(\e/\hat C_R)^{1/q}$,  $\P$-a.s.
$$\sup_{r\in(0,1]}\sup_{|x|\leq R}P_x[\tau_{x,r}\leq\d_R r^\3]\leq \e.$$
\end{cor}
\begin{proof}
	By \Cref{exitmoment} and Markov inequality we have for any $x\in B_R$ and $r\in(0,1]$
		
	$$P_{x}\left[\tau_{x,r} \leq \delta_R r^{\beta}\right]=P_{x}\left[\tau_{x, r}^{-q} \geq\left(\delta_R r^{\beta}\right)^{-q}\right] \leq \hat C_R \delta_R^{q} r^{\beta q-\kappa} \leq \varepsilon.$$
\end{proof}
	
	Now we come to the main result of this subsection, which gives the exit time estimate of large balls.
\begin{prop} \label{exit}
Let $q,p,p',\k$ and $\hat C_R$ be as in \Cref{exitmoment}. Then 
$\P$-a.s. for any $\e\in(0,1/4]$ and any $\3>\kappa/q$ the following holds. Let $R\>1$ and $\d_{2R}:=(\e/\hat C_{2R})^{1/q}$. Then for some $c_{\3,\e}>0$ we have
	for any $t\in(0,R\d_{2R}/(2\beta))$ and $r\in[2{\3t}/{\d_{2R}},R]$,
	$$\sup_{|x|\leq R}P_x[\tau_{x,r}\leq t]\leq \frac{1}{1-2\e}\exp\left(-c_{\3,\e}{(\frac{\d_{2R} r^\3}{t})^{\frac{1}{\3-1}}}\right).$$
\end{prop}

\begin{proof}
	 For any $x\in B_R$ and $r\in(0,R]$ set $\~r=r/K\leq1$ for some $K>2$ to be determined. Let $\theta$ be the shift operator for $\{Y_t\}_{t\>0}$. Define
	$$\tau_0:=0,\quad r_0:=0,\quad \tau_n:=\tau_{Y_{\tau_{n-1}},\~r}\circ\theta_{\tau_{n-1}}+\tau_{n-1},\quad r_n:=|Y_{\tau_n}-Y_0|,\qquad n\geq1$$
and $N:=\min\{n:r_n> r/2\}$. Note that $\{\t_n\}_{n\>0}$ are stopping times w.r.t. the right-continuous filtration generated by $Y$, because for $n\>1$
\begin{align*}
	\t_n=\t_{n-1}+\inf \{s\>0:|Y_{s+\t _{n-1}}-Y_{\t _{n-1}}|>\~r\}=\inf \{s\>\t_{n-1}:|Y_{s}-Y_{\t _{n-1}}|>\~r\}
\end{align*}
and hence by the continuity of sample paths of $Y$ (and induction on that $\t_{n-1}$ is a stopping time)
$$\{\t_n<t\} = \cup_{s\in [0,t)\cap\Q}\{\tau_{n-1}\<s, |Y_s-Y_{\t _{n-1}}|>\~r\}\in \sigma\{Y_s;s\<t\}.$$


By the strong Markov property we have
\begin{align*}
P_x\left[\t_n\leq t, N=n\right]
&\<P_x\left[\max_{0\<i\<n-1}|Y_{\t_i}|<2R,~\#\{i\in \{1,...,n\}:\t_i-\t_{i-1}\<\d_{2R} \~r^\3\}\>n-t/(\d_{2R}\~r^\3)\right]\\
&\<2^n\e^{n-t/(\d_{2R} \~r^\3)}.	
\end{align*}

Note that $N\> K/2$ since the path of $Y$ needs to exit at least $\lceil K/2 \rceil$ balls of radius $\~r$ before achieving $r_n>r/2$, and that $\t_N\<\t_{Y_0,r}$ by $\~r<r/2$. So by the estimates above we get

	\begin{align*}
 P_x[\t_{x,r}\leq t]
&\leq \sum_{n=\lceil K/2 \rceil}^\8 P_x\left[\t_n\leq t,N=n\right]\\
&\leq \sum_{n=\lceil K/2 \rceil}^\8 2^n\e^{n-t/(\d_{2R} \~r^\3)}\\
&=  \frac{\e^{-t/(\d_{2R} \~r^\3)}(2\e)^{\lceil K/2 \rceil}}{1-2\e}\\
&\< \frac{1}{1-2\e}\exp\left(\frac{K}{2}\log(2\e)-\frac{K^\3 t}{\d_{2R}r^\3}\log \e\right).
\end{align*}
Set $K=(\frac{\d_{2R} r^\3}{2\3t})^{\frac{1}{\3-1}}$. The assertion is obvious if $K\<2$, by choosing $c_{\3,\e}\< \frac{1}{2}(2\3)^{\frac{-1}{\3-1}}\log \frac{1}{1-2\e}$. When $K>2$ and $r\>{2\3t}/{\d_{2R}}$ we have $K\>r$ so that $\~r\<1$, then we get  
	$$ P_x[\t_{x,r}\leq t]\<\frac{1}{1-2\e}\exp\left({-c_{\3,\e}(\d_{2R} r^\3/t)^{\frac{1}{\3-1}}}\right)$$
for $c_{\3,\e}=\frac{1}{2}(\3-2)(2\3)^{-\frac{\3}{\3-1}}\log \frac{1}{\e}>0$.	
\end{proof}

\subsection{Liouville heat kernel upper bounds}

We first establish the on-diagonal bound of the Liouville heat kernel at large distances.

\begin{prop} \label{DUE}
For any $\r\in(0,2)$ and $\2\in(0,\2_2)$ we have $\P$-a.s. for any $R>2$ and $t\in(0,1/2]$,
$$\sup_{|x|,|y|< R}p_t^R(x,y)\lesssim_{X,\r,\2} (\log R)t^{-1}\log t^{-1}$$
where $p_t^R$ is the Liouville heat kernel killed upon exiting $B_R$, and
$$\sup_{|x|,|y|< R}p_t(x,y)\lesssim_{X,\r,\2,q,\k} (\log R)t^{-1}\log t^{-1}$$
where $q,\k$ are from \Cref{exitmoment} and $q>2$.
\end{prop}

\begin{proof}
To get the bound for $p_t^{R}(x,y)$, we first show a Faber-Krahn-type inequality (an estimate for the smallest eigenvalue of the generator). For a fixed non-empty bounded open set $U\inn B_R$, let $\lambda_1(U)$ be the smallest eigenvalue of the generator $-\mathcal L_U$ of the LBM killed upon leaving $U$ and 
$G_Uf(x)= (-\mathcal L_U)^{-1}f(x)=\int g_U(x,y)f(y)M(dy)$ where $g_U$ is the Green kernel of the standard Brownian motion killed upon leaving $U$. For $g_U$ we have (see e.g. \cite[Lemma 3.37]{morters2010brownian}) for any $x,y\in U\inn B_R$

\begin{align*}
	g_U(x,y) 
	&\< g_{B_{2R}}(x,y)\\
	&= \frac{1}{\pi}\log\frac{1}{|x-y|}+\frac{1}{\pi}\EX_x\left[\log|W_{T_{2R}}-y|\right]\\
	&\< \frac{1}{\pi}\log\frac{1}{|x-y|}+\frac{1}{\pi}\log(3R)
\end{align*}
where $T_{2R}=\inf\{t\>0:W_t\notin B_{2R}\}$.
We have for $\3>0$
\begin{align*}
	\|G_U1\|_\8 
	&= \sup_{x\in U}\int g_U(x,y)M(dy)\\
	&\lesssim \sup_{x\in U}\int_U \left(\log(3R)+\log\frac{1}{|x-y|}\right)M(dy)\\
	&\lesssim M(U)\left[\log(3R)+\sup_{x\in U}\3^{-1}\int_U \log\frac{1}{|x-y|^\3}\frac{M(dy)}{M(U)}\right]\\
	&\lesssim M(U)\left[\log(3R)+\sup_{x\in U}\3^{-1} \log\int_U \frac{1}{|x-y|^\3}\frac{M(dy)}{M(U)}\right],
\end{align*}
where the last inequality follows from Jensen's inequality. By \Cref{liouvillemeasure}, one can get for any $x\in B_R$
\begin{align*}
	\int_U \frac{1}{|x-y|^\3}\frac{M(dy)}{M(U)}
	&\< 1+\sum_{n=1}^\8\int_{U\cap\{2^{-n}<|x-y|\<2^{-n+1}\}} \frac{1}{|x-y|^\3}\frac{M(dy)}{M(U)}\\
	&\<1+\sum_{n=1}^\8 \frac{2^{\3 n} M(\{|x-y|\<2^{-n+1}\})}{M(U)}\\
	&\< 1+\frac{2^\2\bar C_R}{M(U)}\sum_{n=1}^\8 2^{(\3-\2)n},
\end{align*}
where $\2>0$ is from \Cref{liouvillemeasure}. Choose $\3=\2/2$, and the sum $C_\2=2^\2\sum_{n=1}^\8 2^{-\2n/2}$ $(>1)$ is finite. 
Hence 
\begin{align*}
	\|G_U1\|_\8
	&\lesssim M(U)\left[\log(3R)+\log(1+\frac{C_\2 \bar C_R}{M(U)})\right]\\
	&\lesssim M(U)\left[\log(3R)+\log(C_\2\bar C_R)+\log(\frac{1}{C_\2 \bar C_R}+\frac{1}{M(U)})\right]\\
	&\lesssim \left(\log(3R)+\log(C_\2 \bar C_R )\right)M(U)\log\left(2+\frac{1}{M(U)}\right).\\
\end{align*}
By \cite[Lemma 3.2]{grigor2012two} we know $\lambda_1(U)^{-1}\<\|G_U1\|_\8$ and hence
$$\lambda_1(U)\gtrsim\frac{C_9}{M(U) \log \left(2+\frac{1}{M(U)}\right)}$$
where $C_9^{-1}=\log(3R)+\log( C_\2  \bar C_R )\lesssim_{X,\r,\2} \log(R)$ provided $R>2$.

Now we apply the proof of  \cite[Proposition 5.3]{andres2016continuity}. Let $T^{B_R}_t$ be the semigroup operator associated to the heat kernel $p_t^R$. They obtained that 
$$\|T_{t}^{B_R}\|_{L^{1}(B_R) \rightarrow L^{\infty}(B_R)} \leq m(t)$$
for some function $m(t)$. By following their proof, it is straightforward to check that for $t\in(0,1/2]$
$$m(t)\< 4 C_9^{-1} t^{-1}\log t^{-1}\lesssim_{X,\r,\2} (\log R) t^{-1}\log t^{-1}.$$ 
Hence the bound for $p_t^R(x,y)$ follows. 

To extend the bound to $p_t(x,y)$, we can use Kigami's iteration argument \cite[Lemma 5.6]{grigor2014upper}. Let $Q_t(R):=C_Q(\log R)t^{-1}\log t^{-1}$ where $C_Q=C_Q(X,\r,\2)$ is the constant so that 
$$\sup_{|x|,|y|\leq R}p_t^R(x,y)\leq Q_t(R).$$
Notice that for any $s\in[t/2,t],~\lambda\in[1,4]$ we have 
$$Q_s(\lambda R)\<12 Q_t(R).$$ 
Let $L:=12$ and $\e:=\frac{1}{2L}$. Now choose $R_0=R_0(X,\r,q,\k)>0$ large enough so that for any $R\>R_0$,
$$ \d_{2R}R>\3~~~~\text{and}~~~~2\exp\left(-c_{\3,1/4}{(2\d_{2R} R^\3)^{\frac{1}{\3-1}}}\right)\<\e.$$
This can be done because when $1-2p'/q>0$ we have $\d_{2R}R\to\8$ as $R\uparrow\8$.
Then from \Cref{exit} we get
$$\sup_{|x|\<R} P_x[\t_{0,4R}\<t] \< \sup_{|x|\leq R}P_x[\tau_{x,R}\leq t]\leq\e.$$
Define for $k=0,1,2,...$ the sequences
$$t_k=\frac{1}{2}(1+2^{-k})t,~R_k=4^kR_0,~B_k=B_{R_k}.$$
Let $$\sup_{U}p_t^R:=\sup_{x,y\in U}p_t^R(x,y)$$ for any set $U\inn\C$.
We apply the inequality \cite[Theorem 4.6]{grigor2014upper} to get
\begin{align*}
\sup_{B_k}p_{t_k}^R
&\leq \sup_{B_{k+1}} p_{2^{-(k+2)}t}^{R_{k+1}}+\e\sup_{B_{k+1}} p_{t_{k+1}}^{R}\\
&\leq Q_{2^{-(k+2)}t}({R_{k+1}})+\e\sup_{B_{k+1}} p_{t_{k+1}}^{R}\\
&\leq L^{k+2}Q_{t}({R_0})+\e\sup_{B_{k+1}} p_{t_{k+1}}^{R}
\end{align*}
as long as $R_{	k+1}\<R$. Let $n\>1$, set $R=R_n$ and by iteration we get

\begin{align*}\sup_{B_{0}} p_{t}^{R_n} 
& \leq L^{2}\left(1+L \varepsilon+(L \varepsilon)^{2}+\ldots\right) Q_{t}\left(R_{0}\right)+\varepsilon^{n}\sup_{B_{n}} p_{t_{n}}^{R_n} \\ 
&\<2 L^{2} Q_{t}\left(R_{0}\right)+(L\e)^{n}Q_t(R_0).
\end{align*}
Since $\lim_{n\to \8}p_t^{R_n}(x,y)=p_t(x,y)$ for any $x,y\in B_0$ by \cite[Proof of Theorem 5.1 for unbounded $U$]{andres2016continuity}, let $n\to\8$ and we get 
$$\sup_{B_{0}} p_{t}\leq 2 L^{2} Q_{t}\left(R_{0}\right)=288C_Q(\log R_0)t^{-1}\log t^{-1}.$$
Since $R_0=R_0(X,\r,q,\k)$ can be chosen to be any larger value, it follows that
$$\sup_{|x|,|y|\leq R}p_t(x,y)\lesssim_{X,\r,\2,q,\k} (\log R)t^{-1}\log t^{-1}$$
for any $R>2$.
\end{proof}

Now comes the main result of this paper.

\begin{thm}\label{mainthm}
For any $\r\in(0,2)$, $p>1$, $p'=\frac{p}{p-1}$, $q>2p'$, $\2\in(0,\2_2)$ and $\3>\frac{p}{q}(2+\~\xi(q))$, there exist $c_*, C_*>0$ depending on $X,\r,q,p,\3$, such that $\P$-a.s. for any $t\in(0,1/2]$, $R>2$, and $x,y\in B_R$ with $|x-y|>c_*R^{2p'/q}t$, we have
$$p_t(x,y)\lesssim_{X,\r,\2,q,p,\3} (\log R) t^{-1}\log t^{-1}\exp\left({-C_*\left( \frac{|x-y|^\3}{tR^{2p'/q}}\right)^{\frac{1}{\3-1}}}\right).$$
\end{thm}
\begin{proof} We apply a result in \cite[Theorem 10.4]{grigor2004heat} (see also \cite[Theorem 5.1]{GRIGORYAN20102613}) that, if $U,V$ are non-empty open subsets of $\C$ with $U\cap V= \O$, then for any $(x,y)\in V\times U$,
	$$p_{t}(x, y) \leq \psi^{V}\left(x, \frac{t}{2}\right) \sup _{t / 2 \leq s \leq t} \sup_{v\in\6 V} p_{s}(v, y)+\psi^{U}\left(y, \frac{t}{2}\right) \sup _{t / 2 \leq s \leq t} \sup_{u\in\6 U} p_{s}(u, x)$$
	where $\psi^V(z,s)=P_z[\t_V\<s]$, and $\t_V$ is the first exit time of Liouville Brownian motion from $V$. Note that we can change \enquote{$\esup$} to \enquote{$\sup$} because $p_t(x,y)$ has been proved to have a $(t,x,y)$-jointly continuous version (see \cite[Theorem 1.1]{andres2016continuity}) and $\psi^V(x,t/2)	$ is continuous in $x\in V$ by \cite[Theorem 5.1(ii)]{andres2016continuity}.
	
	Set $r=|x-y|/2$, $V=B_{x,r}$, $U=B_{y,r}$. Applying \Cref{exit} with $\e=1/4$ and $\k=\frac{1}{2}(q\3+p(2+\~\xi(q)))$ (so that $\3>\kappa/q$) leads to 
	$$\psi^{V}\left(x, \frac{t}{2}\right)\vee\psi^{U}\left(y, \frac{t}{2}\right)	\<	2\exp\left({-c_{\3,1/4}(2\d_{2R} r^\3/t)^{\frac{1}{\3-1}}}\right)$$
	provided $\d_{2R}r>\3t$. In particular by \Cref{exitmoment} it is true if $|x-y|>c_*R^{2p'/q}t$ for some $c_*=c_*(X,\r,q,p,\3)>0$.

	Furthermore, by \Cref{DUE} we have
	$$\sup _{t / 2 \leq s \leq t} \sup_{v\in\6 V} p_{s}(v, y)	\vee	\sup _{t / 2 \leq s \leq t} \sup_{u\in\6 U} p_{s}(u, x)\lesssim_{X,\r,\2,q,\k} (\log R)t^{-1}\log t^{-1}.$$
	Hence
	\begin{align*}
 p_t(x,y)
&\lesssim_{X,\r,\2,q,p,\3} (\log R)t^{-1}\log t^{-1}\exp\left({-\frac{1}{2} c_{\3,1/4}(\d_{2R} |x-y|^\3/t)^{\frac{1}{\3-1}}}\right)\\
&\lesssim_{X,\r,\2,q,p,\3} (\log R)t^{-1}\log t^{-1}\exp\left({-C_*\left( \frac{|x-y|^\3}{tR^{2p'/q}}\right)^{\frac{1}{\3-1}}}\right)
\end{align*}
for some $C_*=C_*(X,\r,q,p,\3)>0$.
\end{proof}

\begin{corollary} \label{Cor1}
	Under the setting of \Cref{mainthm}, set $R=|x|\vee|y|\vee 2$. If in addition $R\<c|x-y|$ for some $c>0$, then
$$p_t(x,y)\lesssim_{X,\r,\2,q,p,\3} (\log R)t^{-1}\log t^{-1}\exp\left({-\~C_*\left( \frac{|x-y|^{\3-2p'/q}}{t}\right)^{\frac{1}{\3-1}}}\right)$$
for some $\~C_*=\~C_*(X,\r,q,p,\3,c)>0$.
\end{corollary}

\begin{corollary}
	Under the setting of \Cref{mainthm}, for any $t\in(0,1/2]$ and $R_0>1$, there exists $R_1 = R_1(c_*,R_0,p,q)>R_0$ such that for any $y\in B_{R_0}$, $x\notin B_{R_1}$, we have 
	$$p_t(x,y)\lesssim_{X,\r,\2,q,p,\3} \exp\left({-0.5\~C_*\left( \frac{|x-y|^{\3-2p'/q}}{t}\right)^{\frac{1}{\3-1}}}\right)$$
	for some $\~C_*=\~C_*(X,\r,q,p,\3)>0$.
\end{corollary}
\begin{proof}
	Choose $R_1$ such that $R_1> (c_*/2+1)^{\frac{q}{q-2p'}}\vee (R_0+1)^{\frac{q}{2p'}}\vee (2R_0)$ (recall that $q-2p'>0$). Then for any $y\in B_{R_0}$, $x\notin B_{R_1}$ we have $2|x-y|\>2(|x|-R_0)\>R$ and $|x-y|>c_*R^{2p'/q}t+1$, where $R=|x|\vee |y|$. The result follows from \Cref{Cor1} with $c=2$ by absorbing $(\log R)t^{-1}\log t^{-1}$ into the exponential.
\end{proof}

\begin{cor}
	Liouville Brownian motion is $C_0$-Feller in the sense that $T_t$ is a positive contraction strongly continuous semigroup on $C_0$, where $C_0$ is the space of continuous functions on $\C$ vanishing at infinity.
\end{cor}

\section{Appendix: A simple proof of the Feller property}
In this section, we give a simple proof of the $C_0$-Feller property without using heat kernel estimates.

We slightly change the notation.
Let $X$ be a whole plane (massive) Gaussian free field defined on some probability space $\Omega$ with the law denoted by $\P^X$, and $B$ be a Brownian motion defined on another probability space. Let $\P^B_x$ be the law of Brownian motion starting from $x\in\C$. By making the product space and setting $\P_x=\P^X\otimes\P^B_x$, then $X$ and $B$ are independent under $\P_x$. We use $\EX_x$, $\EX_x^B$ and $\EX^X$ to mean taking expectation under $\P_x, \P^B_x$ and $\P^X$ respectively. When $x=0$, we drop the subscript. 

Let $F$ denote the PCAF of $B$ whose Revuz measure is the Liouville measure and $\bar F$ be the inverse of $F$, i.e., $\bar F(t)=F^{-1}(t)=\inf\{s\geq 0:F(s)>t\}$. 
We denote the Liouville Brownian motion by $Y_t=B_{\bar F(t)}$, and define the running supremum $Y^*_t:=\max_{s\leq t}|Y_s-Y_0|$.

Let $\D_R$ be the open disk with center 0 and radius $R>0$ and $\bar\a_R(dx)$ be the uniform probability measure on the circle $\6\D_R$ . For a finite set $S$, we use $|S|$ to denote the number of elements in $S$. 

The following discussion is for $\P^X$-a.e. element of $\Omega$. Recall that it has already been proved that $T_t$ maps $C_b$ to $C_b$, where $C_b$ is the set of bounded continuous functions on $\C$.
Fix $t>0$. To show $T_t$ maps $C_0$ to $C_0$, it is enough to show that for any $R>0$,
$$\lim_{x\to\8}\P^B_x[Y_t\in\D_R]=0.$$
Indeed, for any $f\in C_0$ and $\e>0$ there is a continuous function $f_\e\in C_K$ with compact support such that $\|f-f_\e\|_\8<\e$. Choose $R$ large enough so that the support of $f_\e $ is contained in $\D_R$, then $|T_t f(x)|\leq\e+\lVert f_\e\rVert_\8\P^B_x[Y_t\in\D_R]$. Let $x\to\8 $ and then $\e\to 0$, and we get $\lim_{x\to\8}T_t f(x)=0$.

Now fix $R>0$ and $t>0$. Define $g(x)=g(x,X):=\P^B_x[Y^*_t\geq |x|-R]$.

~\\

\begin{lem}\label{rot}
 Let $\0\in\C$ and $|\0|=1$. Then $g(\0 x)$ and $g(x)$ have the same law under $\P^X$. In particular we have $\EX^Xg(x)=\EX^Xg(|x|)$.
\end{lem}
\begin{proof}
	Let $X^\0=X(\cdot/\0)$ and $B^\0=\0B$. First we show that $\P^X$-a.s., $F_t(X,B)=F_t(X^\0,B^\0)$ $P^B_x$-a.s. for any $x\in\C$. Indeed, we have 
	$$F^n(t)=\int_0^t \exp \left(\gamma X_{n}(B_s)-\frac{\gamma^{2}}{2} \EX^X\left[X_{n}(B_s)^{2}\right]\right) ds=\int_0^t \exp \left(\gamma X_{n}^\0(B_s^\0)-\frac{\gamma^{2}}{2} \EX^X\left[X_{n}^\0(B_s^\0)^{2}\right]\right) ds.$$
	Let $n\to\8$ and by the uniqueness of the limit we have $F_t(X,B)=F_t(X^\0,B^\0)$, and consequently $\bar F_t(X,B)=\bar F_t(X^\0,B^\0).$ 
	
	Notice that 
	\begin{align*}
	g(x,X)&=\P^B_x[Y^*_t\geq |x|-R]\\
		&=\P^B_x[\max_{s\leq t}|B_{\bar F_s(X,B)}-B_0|\geq |x|-R]\\
		&=\P^B_x[\max_{s\leq t}|\0 B_{\bar F_s(X^\0,B^\0)}-\0 B_0|\geq | x|-R]\\
		&=\P^B_x[\max_{s\leq t}|B^\0_{\bar F_s(X^\0,B^\0)}-B^\0_0|\geq | x|-R]\\
		&=\P^{B}_{\0 x}[\max_{s\leq t}|B_{\bar F_s(X^\0,B)}-B_0|\geq |\0 x|-R]\\
		&=g(\0 x,X^\0).
	\end{align*}
Since $X$ and $X^\0$ have the same law under $\P^X$ by the rotation invariance of the covariance function $G_m$, we see that $g(\0 x)$ and $g(x)$ also have the same law under $\P^X$. Now fix $x\in\C$, choose $\0$ such that $\0 x=|x|$, take the expectation, and we get $\EX^Xg(x)=\EX^Xg(|x|)$.
\end{proof}

\begin{lem}\label{lem}

Let $x_n=n\in\C$. Then $\lim_{n\to\8}\EX^X g(x_n)=0$.
\end{lem}
\begin{proof}
For $x\in\C\-\{0\}$ and $\e>0$, set $s=\e|x|^2$, and we get
	\begin{align*}
	g(x)
	&\<\P^B_x\left[Y^*_t\geq |x|-R, \bar F(t)\leq s\right]+\P^B_x\left[\bar F(t)>s\right]\\
	&\leq \P^B\left[\max_{l\leq s}|B_l|\geq |x|-R\right]+\P^B_x\left[\frac{\bar F(t)}{|x|^2}>\e\right]\\
	&= \P^B\left[\max_{l\leq 1}|B_l|\geq \frac{|x|-R}{\sqrt \e |x|}\right]+\P^B_x\left[\frac{\bar F(t)}{|x|^2}>\e\right].\\
	\end{align*}

Thus by the translation invariance of the law of $X$ we have
$$\EX^X g(x_n)\leq \P^B\left[\max_{l\leq 1}|B_l|\geq \frac{1}{2\sqrt \e}\right]+\P[\bar F(t)/n^2>\e]$$
provided $n$ is large enough so that $R/n\<1/2$.
Now let $n\to\8$, then $\e\to0$ and we get the desired result.
\end{proof}

Now we are ready to prove that $T_t$ is Feller.

\begin{thm}
$\P^X$-a.s., $T_t$ maps $C_0$ to $C_0$.	
\end{thm}

\begin{proof}
By Lemma \ref{rot} we have 
$$\EX^X\int g(x) \bar\a_n(dx)=\int\EX^X g(n)\bar\a_n(dx)=\EX^Xg(n).$$
By Lemma \ref{lem} we get 
$$\lim_{n\to\8}\EX^X\int g(x) \bar\a_n(dx)=\lim_{n\to\8}\EX^Xg(n)=0.$$
Thus there is a subsequence of $n$ along which $\int g(x) \bar\a_n(dx)\to 0$ $\P^X$-a.s.. Then for any $\e>0$ and $\d>0$, there is some $n>R$ sufficiently large such that $$\bar\a_n(\{x\in\6\D_n:g(x)>\e\})<\d.$$ 

Set $S_n=\{x\in\6\D_n:g(x)\leq\e\}$, then we have $\bar\a_n(S_n^c)\leq\d$.

Let ${\tau_n}=\inf\{s>0:Y_s\in\6\D_n\}$, then $|Y_{\tau_n}|=n$. When $|x|>n>R$, using the strong Markov property (see, e.g., \cite[Proposition 3.4]{grigor2017localized}), we have
\begin{align*}
\P^B_x[Y_t\in\D_R]
&=\P^B_x[Y_t\in\D_R,{\tau_n}<t]\\
&=\int_{\{\t_n<t\}}\P^B_{Y_{\tau_n}(\w)}[Y_{t-{\tau_n(\w)}}\in\D_R]\P_x^B(d\w)\\
&\leq\EX^B_x\left[P^B_{Y_{\tau_n}}[Y^*_t\>n-R]\right]\\
&\leq\EX^B_x\left[g(Y_{\tau_n}), Y_{\tau_n}\in S_n\right]+\P^B_x[ Y_{\tau_n}\in S_n^c]\\
&=\int_{S_n} g(z)\mu_{x,n}(dz)+\mu_{x,n}(S_n^c)
\end{align*}
where $\mu_{x,n}(dz)=\P^B_x[ Y_{\tau_n}\in dz]$, which is the harmonic measure of the Brownian motion viewed at $x$. 

We claim 
$\mu_{x,n}\to \bar\a_n$ in total variation as $x\to\8$. Indeed. Let $\phi(z)=n^2z/|z|^2$. Notice that $\phi$ is analytic on $\C\-\{0\}$ and $\phi|_{\6\D_n}$ is the identity map. We have 
$$\mu_{x,n}(dz)=\P^B_x[ B_{\bar F(\tau_n)}\in dz]=\P^{B'}_{x'}[ {B'}_{{\tau_n}'}\in dz]$$
 where $x'=\phi(x)$, $B'=\phi(B)$ (which is a time-change of a Brownian motion) and ${\tau_n}'=\inf\{s>0:B'_s\in\6\D_n\}$. Thus
 $$\mu_{x,n}(dz)=\mu_{x',n}(dz)=p_n(x',z)\bar\a_n(dz)$$
 where $p_n(x',z)$ is the Poisson kernel on $\6\D_n$.
Hence
 $$\|\mu_{x,n}-\bar\a_n\|_\text{total variation}=\int|p_n(x',z)-1|\bar\a_n(dz)\to0$$
 as $x'\to 0$ ($x\to\8$).
 So we have $\P^X$-a.s.
\begin{align*}
 \limsup_{x\to\8}\P^B_x[Y_t\in\D_R]
&\<\int_{S_n} g(z)\bar\a_n(dz)+\bar\a_n(S_n^c)\\
&\leq \e+\d.
\end{align*}

Let $\e\to 0$ and $\d\to 0$, and combining the discussion at the very beginning of this section, we complete the proof.
\end{proof}

\section{Acknowledgements}
The author sincerely thanks his advisor Professor Zhen-Qing Chen for giving the topic, useful references and helpful discussions, without which this paper could not have been done. 

The author also thanks the anonymous referee for a very careful reading of the manuscript and helpful comments.

\bibliography{LHKlargedistance}
 \bibliographystyle{plain}

\begin{flushright}
Yang Yu\\
Department of Mathematics\\
University of Washington\\
Seattle, WA 98195\\
USA\\	
\end{flushright}

\end{document}